\documentclass[11pt, leqno]{amsart}

\usepackage{amssymb}
\usepackage{amsmath}
\usepackage{amscd}
\usepackage{float}
\usepackage{graphicx}
\usepackage{amsfonts}
\usepackage{color}
\newtheorem{theorem}{Theorem}

\newtheorem{proposition}{Proposition}
\newtheorem{definition}{Definition}
\newtheorem{remark}{Remark}

\def\F{\mathcal F}
\def\T{\mathcal T}
\def\E{\mathcal E}
\def\A{\mathsf{A}}
\def\B{\mathsf{B}}
\def\L{\mathsf{L}}

\begin{document}

\title{TIED LINKS}

\author{Francesca Aicardi}
 \address{The Abdus Salam International Centre for Theoretical Physics (ICTP),  Strada  Costiera, 11,   34151  Trieste, Italy.}
 \email{faicardi@ictp.it}
 \author{Jes\'us Juyumaya}
 \address{Instituto de Matem\'{a}ticas, Universidad de Valpara\'{i}so,
 Gran Breta\~{n}a 1111, Valpara\'{i}so, Chile.}
 \email{juyumaya@gmail.com}

\dedicatory{ To the  memory of  Slavik  Jablan}

\thanks{
This work was  started  during a Visiting Fellowship of the second author at the ICTP.
The second author has been partially supported  by Fondecyt 1141254  and by
the European Union (European Social Fund  ESF) and Greek national funds through the Operational Program
  \lq Education and Lifelong Learning\rq\ of the National Strategic Reference Framework (NSRF)  Research Funding
Program: THALES: Reinforcement of the interdisciplinary and/or interinstitutional research and innovation.
}


\date{}

\begin{abstract}
In  this paper we  introduce the {\it tied  links}, i.e. ordinary links provided with some
\lq ties\rq\
between  strands.    The  motivation   for introducing  such  objects originates from  a diagrammatical interpretation of the defining generators of  the so--called algebra  of  braids  and  ties; indeed,  one  half of  such  generators  can  be  interpreted as the  usual  generators   of the braid algebra,  and the  other half can be  interpreted  as  ties  between consecutive  strands; this interpretation leads to the definition of {\it tied braids}.    We  define
 an  invariant polynomial for  the  tied  links   via  a  skein  relation.  Furthermore,  we  introduce the monoid of tied braids and we prove the corresponding theorems of    Alexander and   Markov for  tied  links. Finally, we prove that the invariant of tied links    we  defined   can be   obtained also by using the Jones recipe.
\end{abstract}

\maketitle

\section*{Introduction}
Since  the  seminal  works  of   Jones   \cite{joAM}   on  the famous  invariants  for  classical  links,  now called    the Jones polynomial, several  other classes of knotted objects have  been proposed, singular links \cite{ba, bi, sm}, framed links \cite{ks}  and virtual knots \cite{kau},  among others.  This paper introduces and studies  a   new  class of knotted objects, the  { \sl tied  links}. Tied links are the closure of   {\it tied braids}; tied braids
come  from a diagrammatical interpretation of the  defining generators of the so--called algebra of braids and ties, in short bt--algebra,   which appeared the first time in \cite{ju} and \cite{AJ} and was studied in different contexts in  \cite{rh}, \cite{baART} and \cite{AJtrace}. The bt--algebra was firstly introduced    as a tool to study the
representation theory  of the Vassiliev algebra (see \cite[Section
2.1]{AJ}) and consequently  to construct new representations of the braid group.
Also, in \cite{AJ} it is shown  that this algebra can be Yang--baxterized in
the sense of Jones.

\smallbreak

Let $n$ be a positive integer, the bt--algebra ${\mathcal E}_n$ is  a 1--parameter unital associative  algebra,  defined   through a presentation with  generators $T_1,\ldots ,T_{n-1}$ and $E_1, \ldots , E_{n-1}$ and certain relations, for details see Definition \ref{btalgebra}.  The defining relations of ${\mathcal E}_n$ make evident that the generators $T_i$'s can be interpreted diagrammatically  as the usual braids generators. On  the contrary, an immediate diagrammatical interpretation of the generators $E_i$'s is not evident. However, in \cite{AJ} we proposed the interpretation of  the generator $E_i$ as a tie  that  connects the $i$--strand with the $(i+1)$--strand, thus making  the bt--algebra  a   diagrams algebra. In \cite{AJtrace} we have proved that the bt--algebra supports a Markov trace. Consequently, using this Markov trace, the classical theorems of Markov and Alexander for classical links,  and certain representation of the classical braid group   in the bt--algebra, w
 e have defined an invariant, at three parameters, for  classical links; an analogous route yields an invariant for singular links; see \cite{AJtrace} for details. Now, we    associate   to the bt--algebra  the {\it tied braid monoid},  which is defined essentially as  the monoid  having a  presentation analogous to the bt--algebra, but taking into account only the monomial  relations. We call tied braids  the elements of this monoid and   tied links   the closure of the tied braids.
Notice that the  classical links can be regarded as  tied links since  the braid group can be considered naturally as a submonoid of the tied braid monoid.

\smallbreak

This paper is largely  inspired by our work \cite{AJtrace} on the bt--algebra. More precisely, we  introduce and study the tied braid monoid and the tied links. In particular, we
define an invariant for tied links which \lq contains\rq\ the invariants for classical links defined in \cite{AJtrace}. This  invariant  is defined in two ways:  one  by a skein relation and the other   one by the Jones recipe, that is, via    normalization and  rescaling of the composition of a representation of the tied braid monoid in the bt--algebra with the Markov trace defined on it.

\smallbreak

This paper is  organized  as  follows.  In  Section  1,  we  define  the tied  links,  their  diagrams  and  their isotopy  classes.  In  Section  2,   we prove the first main result (Theorem \ref{the1}) of the paper, that is, the construction, via skein relation, of  an  invariant   polynomial for  tied  links, which we shall denote $\mathcal{F}$.  The  proof of  the existence  of such  invariant is an adaptation of the proof given in \cite{Lic}  for the Homflypt polynomial for  classical  knots.  Further, we  provide  the  value of  the invariant  on  simple  examples  of  tied  links.     Section 3, is devoted to the  understanding  of the tied links   through the introduction of a new object, that we  call the monoid of tied braids (Definition \ref{monoidtb}); also, in this section  we prove the analogous of the Alexander theorem  (Theorem \ref{alexandertb}) and of  the Markov theorem (Theorem \ref{markovtb}) for tied links.
  Section 4,  has as goal the construction of the invariant  $\mathcal{F}$ through the Jones recipe, i.e.,   by the composition of the natural representation of the monoid of tied braids in the bt--algebra  with  the  trace  on  the  bt--algebra defined in \cite[Theorem 3]{AJtrace}; this construction  is done in Theorem \ref{Ftrace}. Finally, Section 5
 is  devoted  to  the  computer  algorithm   to  obtain the  value of  $\mathcal{F}$  for any  tied  link put in  the  form  of  a  closed tied  braid;   the algorithm   calculates  the  trace of  the  tied  braid and  then  the polynomial  for  the tied  link  via  the Jones recipe.

\smallbreak
It  is  a  pleasure  to  dedicate  this  work to the  memory of Slavik Jablan,  who since 2010     showed interest  in  our algorithm   to calculate the invariant  (for classical links) ${\mathcal F}$, here presented, and  whose  existence  was at  that  time  only  conjectured.   We will  always be grateful   to Slavik,
in particular for  his  personal way of  being a mathematician  as  an authentic   \lq truth lover\rq. Slavik always enjoyed  sharing his  knowledge  and achievements with  extreme  generosity.
\smallbreak

\section{Tied  links}\label{S1}

 In this  section  we  introduce  the  concepts  of   tied  links  and  of  their diagrams. In  fact, a tied link is a  link  whose set of components is    subdivided  into  classes.
   Our   Definitions 1  and 3   are motivated by the concept of a {\em tie} as a notational device to indicate that two components of a link belong to the  same  class.
   The tie is an arc drawn from a point on one component to a point on the other or the same component. It will be indicated by a wavy line and it is part of the formalism that the way
   this arc is embedded in three--space is not relevant.

\begin{definition}\rm
{\it A tied  (oriented) link    $L(P)$  with  $n$  components} is  a  set $L$ of $n$ disjoint  smooth (oriented) closed   curves embedded  in $S^3$, and  a  set $P$ of
{\sl ties}, i.e., unordered pairs of points $(p_r,p_s )$ of such  curves  between which  there is   an arc  called a  tie. Ties are depicted as springs connecting pairs of points lying on the curves. The tie is a notational device, not an embedded arc. Arcs can cross through it.
 We shall denote  $\T$ the set of oriented  tied links.
 \end{definition}
\begin{remark}\label{LCT}\rm
If $P$ is the empty set, the tied link
  $L(P)$  is nothing else but the classical link $L$.
\end{remark}
\begin{definition}\rm
{\it A  tied link diagram}  is  like the  diagram of a  link, provided  with   ties,  depicted as  springs   connecting  pairs of points
lying on  the  curves.
\end{definition}

  Observe  that  the  set  $P$ defines  a  partition  of  the  set  of  the  components  of  $L$ into  classes in this  way:  if a  tie connects two  components,  then  these  components
belong  to  the  same  class.  Observe  also  that  the  sets of  components  of  two  isotopic  links  are  equal,  and  can  be  identified.

\begin{figure}[H]
\centering
\includegraphics{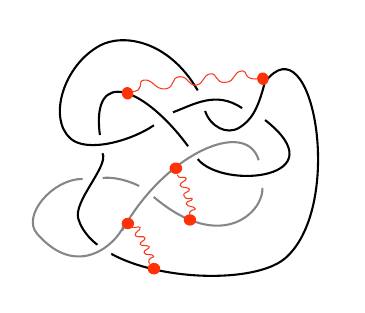} \caption{The  diagram of a tied link.}
\end{figure}

\begin{definition}\label{isotop} \rm Two oriented tied  links   $L(P)$  and  $L'(P')$  are  {\sl tie--isotopic}  if:

\begin{enumerate}
\item  the links $L$  and  $L^{\prime}$  are  ambient
isotopic

\item   The  sets  $P$  and  $P'$  define  the  same  partition  of  the  set of components  of $L$  and  $L'$.

         \end{enumerate}
\end{definition}

In  other  words  the tie--isotopy says that,  when remaining    in the  same equivalence class, it  is allowed  to   move  any tie between two  components letting  its extremes    move
 along the   two  whole components.   Moreover,  ties  can  be destroyed or  created  between  two components,  provided that these  components   either  coincide,
 or  belong  to  the  same  class.

\begin{definition}\rm Two  components  will  be  said  {\sl tied} together,   if  they  belong  to  the  same  class, i.e.,  if   between  them    a  tie already  exists or  a
tie can  be  created.
\end{definition}

  Examples  of  diagrams of  tie  isotopic   tied links  are  given  in   Figure  \ref{TL2}.  The  links  have  three components  D, G  and B  (respectively for  Dark,  Green  and  Blue).
  The   three  links are  evidently ambient  isotopic.   The  three  components are  all  tied  together, i.e.,  they belong  to  a  sole  class.   Indeed,  in  the  first  two  diagrams  D is  tied  with  G  and with  B.
  Therefore  G results to  be  tied with   B.  A  tie  between G  and  B  is in  fact present  in  the  third  diagram.
   Other  ties  can  be  destroyed,  i.e.,    the  tie  between D  and  itself  in  the  first  two diagrams,  the
     tie  between  G  and  itself  in  the   first diagram,  and   the  second  tie  between  D  and  G    the  second  diagram.

\begin{figure}[H]
\centering \includegraphics{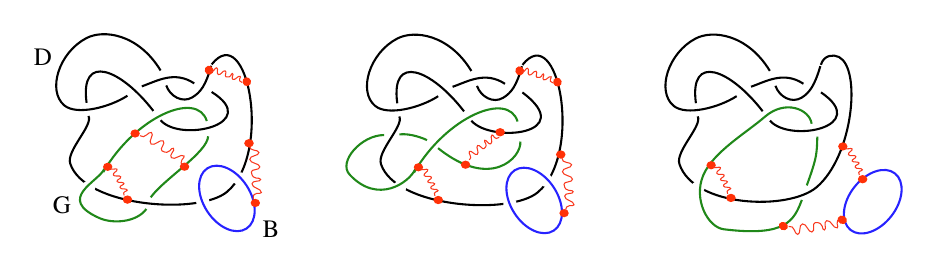}
\caption{Diagrams of tie--isotopic tied links.}\label{TL2}
\end{figure}

Since  the  Reidemeister moves  preserve  the link  components, it is  evident  that  the  diagrams  of  two tie--isotopic  links  can  be  transformed  one  into  the  other by
means  of the  usual  Reidemeister  moves:  the extremes of  the ties  can  be always shifted, so  that  the  ties  result  to  be outside  the ball into  which  each  move takes place.
  Remember that the ties can  freely move (provided only that the  endpoints  move continuously  on  the curves), overstepping  other  ties  as  well  as  the  link  curves.

\begin{definition} \rm  We  call {\it essential}  a  tie between  two  components,  if  it  cannot  be  destroyed (i.e. by  removing  it,  the  components become  untied). Therefore,  an  essential  tie  always exists    between  distinct components.
\end{definition}

\begin{remark} \rm A  tied link  with $n$ components is  therefore  nothing else but  a  link  where  the  components are  colored  with  $m\le n$  distinct  colors \cite{Ai};  two
components have the  same color  if  they  are  tied  together.   However, we deal with diagrams  that  are made  by links  with  ties,  and  that may look arbitrarily complicated.
\end{remark}

\section{An invariant for  tied  links}\label{Sec2}

In order to define our invariant for tied links, we need to fix   some notations. From now on we fix three   indeterminates $u$, $z$ and $w$, and we set $A={\Bbb C }(u,z,w)$.

Let us  denote  by  $\T$    the set of oriented tied  links diagrams. Notice that  an  invariant of  tied  links is  a  function  from $ \T $ in $A$  that  takes one constant  value  on each class  of tie--isotopic links.

\smallbreak

 \noindent {\bf Notation.}    In  the  sequel,   if  there  is  no risk  of  confusion,    we  indicate by $TL$  both the  oriented   tied  link  and   its  diagram.

\subsection{}

 The following  theorem  is  a  counterpart   of  the theorem stated in    \cite[page 112]{Lic}   for  classical  links.

 \begin{theorem} \label{the1} There  exists  a  function    $\mathcal{F}:\T \rightarrow   A$,     invariant  of oriented tied  links,    uniquely  defined  by  the   following  three  conditions  on  tied--links diagrams:
\begin{itemize}
\item[I] The  value of  $\mathcal{F}$ is  equal to 1  on  the  unknotted circle (no  matter if  tied  with itself)

\item[II] Let  $TL$   be  a tied  link.  By  $TL^o$  we  denote  the   tied   link  consisting of $TL$  and  the unknotted  circle (no  matter if  tied  with itself),  unlinked to $TL$.
Then
$$
\mathcal{F}(TL^o)=  \frac {1}{wz} \mathcal{F}(TL)
$$
\item[III] Skein  rule:  let $TL_+, TL_-, TL_\sim, TL_{+,\sim}$  be  the  diagrams  of  tied  links,  that  are  identical    outside   a  small  disc into  which enter two  strands, whereas  inside the disc the  two strands look  as
   shown  in Fig. 3.  Then  the  following  identity  holds:
$$
 \frac{1}{  w}\mathcal{F}( TL_ +)-w \mathcal{F} (TL_-) = \left(1-u^{-1}  \right) \mathcal{F} (TL_\sim) + \frac{1}{w} (1-u^{-1} ) \mathcal{F}(TL_{+,\sim}).
$$

\end{itemize}
\end{theorem}

\begin{figure}[H]
\centering
\includegraphics{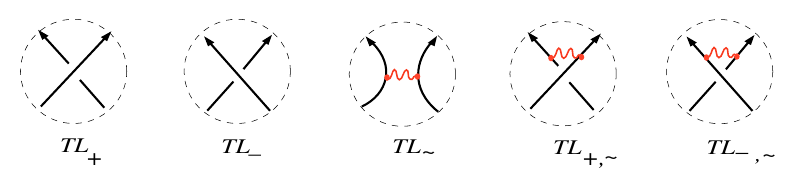}
\caption{ The  discs,  where   $TL_+, TL_-$, $TL_{\sim}$  and  $TL_{+,\sim}$   are not tie-isotopic.}\label{ties3}
\end{figure}

\begin{remark}\label{othersSkein}
\rm The following   three skein  rules,  all  equivalent  to  the  skein  rule  above,   will   be  used in  the  sequel. The  first  one is  obtained   from   III, simply  adding a  tie between  the  two  strands  inside  the  disc. The  rules  Va,b  follow   from  III  and  IV.

 \begin{enumerate}
 \item[IV]
 \[ \frac{1}{u w }\mathcal{F}( TL_ {+,\sim})- w \mathcal{F} (TL_{-,\sim}) = (1-u^{-1}) \mathcal{F} (TL_\sim) \]

\item[Va]
 \[ \frac{1}{ w }\mathcal{F}( TL_ {+})= w \left[\mathcal{F}(TL_-) +\left( {u}-1 \right) \mathcal{F}(TL_{-,\sim})\right] + \left(u-1\right) \mathcal{F} (TL_\sim) \]

\item[Vb]
\[   w  \mathcal{F}( TL_ {-})= \frac{1}{ w } [\mathcal{F}(TL_+) +(u^{-1}-1) \mathcal{F}(TL_{+,\sim})] + ({u^{-1}}-1) \mathcal{F} (TL_\sim).\]
\end{enumerate}
\end{remark}

\begin{proof}
Theorem \ref{the1} is  proved by  the  same  procedure used in  \cite{Lic}.  We  will outline  the  parts  where  the  presence  of  ties modifies  the  demonstration.

  According  to \cite{Lic},   the  fact  that  the skein  rules,  together  with  the  value  of  the  invariant  on  the  unknotted  circle,  are    sufficient  to  define  the  value  of  the  invariant  on  any tied   link  is  proved  as  follows.

 Let  $\T^n$  be the  set  of  diagrams  of  tied  links  with $n$  crossings,  and  $TL\in \T^n$. Ordering  the  components  and  fixing a  point     in  each  component,  for  every  diagram   $TL$   an  associated  standard ascending diagram   $TL'$ is constructed.  This ascending diagram is obtained by  starting  at the  base point  of the  first   component  of $TL$,  and proceeding  along  that  component, changing  the overpasses  to underpasses (where  necessary) so  that  every  crossing  is   first  encountered  as  an underpass. Continue  from  the  base point of  the  second  and of all subsequent  components in  the  same  way. This  process  separates  and  unknots the  components.    The  diagrams   $TL$ and   $TL'$  are thus  identical  except    for a  finite  number  of  crossings,  here  called  \lq deciding',  where  the signs  are  opposite. Furthermore,  we  define  here  $\widetilde{TL}'$,  adding  to  $TL'$  a  tie between  the  strands near to {\it each}  dec
 iding crossing.   $TL'$   and  $\widetilde{TL}'$ are  by  construction    collections  of  unknotted  and  unlinked  components;  $TL'$ has   the  same  ties  as  $TL$,  $\widetilde{TL}'$  may  have  more  ties.  The  procedure  defining  $TL'$ allows us to  get  an  ordered   sequence  of deciding crossings, whose  order  depends  on  the ordering  of  the  components,    and  on the  choice of the base points.

 The  induction  hypothesis   states  that  we  have         a function  $\mathcal{F}: \T^n \rightarrow A$,  which satisfies   relations I--III. This  function is  independent of the  ordering    of  the  components,   independent of  the  choices  of  the  base  points, and invariant  under  Reidemeister  moves  which  do  not  increase  the  number  of  crossings  beyond  $n$.  Moreover,  also  by  induction hypothesis, the  value  of  $\mathcal{F}$  on  any  diagram  with  $n$  crossings of  the  tied link
 $TL_0^{c,m}$, consisting   of  $c$  components  unknotted  and  untied,    connected by  $m$  essential     ties ($m\le c-1$),    is   equal to
$t^{m}/ (wz)^{c-1}$,   where
\begin{equation}\label{tofw}
t:=\frac{z(u w^2-1)}{(1-u)}
\end{equation}
(observe  that  these  values of $\mathcal{F}$ are  independent  of $n$).

  One  starts  with  zero   crossings: the  tied  link  is  thus  a  collection  of  $c\ge 1$ curves  unknotted  and  unlinked,  with  $m \le c-1$  essential  ties  between  them.  The  value of  $\mathcal{F}$ on  such tied  link  is given   by   $t^{m}/(wz)^{c-1}$.

 Now, let  $TL$  be  in  $\T^{n+1}$.   If  $TL$ consists   of  $c$  components  unknotted  and  untied,    connected by  $m$  essential     ties ($m\le c-1$),  then we define
 \begin{equation}\label{eq2} \mathcal{F}(TL)=   \frac{t^{m}}{(wz)^{c-1}}.
  \end{equation}   Otherwise,
 consider  the  first    deciding  crossing $P$.   If  in  a  neighborhood of  $P$  the  tied  link  looks  like $TL_{+,\sim}$ (or  $TL_{-,\sim}$), then we
  use  skein  rule  IV    to  write   the  value  of  $\mathcal{F}$  in  terms  of     $TL_{-,\sim}$ and  $TL_{\sim}$ (respectively, $TL_{+,\sim}$ and $TL_{\sim}$). If  in  a  neighborhood of  $P$  the  tied  link  looks  like $TL_+$  (respectively, $TL_-$), then we  use  skein  rule V-a  (resp.,  V-b)  to  write    the  value  of  $\mathcal{F}$ in  terms of  the    value of $\mathcal{F}$  on  the  tied  links    $  TL_{-},  TL_{-,\sim}$ and  $TL_{\sim}$ (respectively, $  TL_{+},  TL_{+,\sim}$ and  $TL_{\sim}$).  Observe  that  if  the  tied  links $TL_{\sigma}$  or  $TL_{\sigma, \sim}$ ($\sigma=\pm$)  coincide  with  the  original  tied  link  in a  neighborhood  of  the  crossing  $P$,    then $TL_{-\sigma}$ and  $TL_{-\sigma,\sim}$ in the  neighborhood of  the same  crossing coincide with  the associated   tied  link  $TL'$ or  $\widetilde{TL}'$.  On  the  other  hand, $TL_{\sim}$ represents  a  tied link diagram  with  $n$  crossings,  for  which  the  value  of  $\mathcal{F}$  i
 s  known,  and  invariant  according  to  the  induction  hypothesis. Then  we  apply
  the  same  procedure  to the  second  deciding crossing,  which  is  present  in  all the  diagrams,  obtained  by  the  application  of  the  skein  rules,  and that  results to have  $n+1$ crossings,  and  so  on. The  procedure  ends  with  the  last  deciding  crossing, thus obtaining      unlinked and  unknotted tied  links, with  $n+1$ crossings,  where  the value of  $\mathcal{F}$ is  given   by (\ref{eq2}), depending  only  on  the  number  of  components  and  the  number  of  essential  ties.

Observe  that   the  skein  rule  IV  could   be  avoided:  this  should extend the  procedure.

  It  remains  to  prove  that:

  \begin{enumerate}
  \item[({\it i})] the  procedure  is  independent  from  the  order of  the  deciding points
  \item[({\it ii})]  the  procedure  is independent  from  the order of  components,   and from the choice  of  base--points
  \item[({\it iii})] the  function   is  invariant  under  Reidemeister  moves.
\end{enumerate}

Following  the  proof done  in \cite{Lic} for  classical  links,  we  observe  that the  proofs  of points ($i$),  ($ii$),  and  ($iii$) can also be  done  in an  analogous way   in     presence  of ties.  Of  course,   every  time  a  skein  rule  is  used,  we  have  to  pay  attention   to    all tied links  involved   (the skein  relation for  $\F$ involves four  tied links  diagrams,  whereas the  skein  relation  for  the classical  Homflypt polynomial involves only three).  We  write  here  the  proof  of  statement ($i$)  as  an  example.  The  proofs  of the other statements  are similar.

The  proof  of statement ({\it i})  consists    in
  a  verification  that  the  value  of  the  invariant  does  not  change  if  we   interchange  any  two  deciding  points  in  the  procedure of  calculation.  So,  let  $TL$  be the  diagram  of  a  tied  link      and let  $p $  and  $q$  the  first two  deciding crossings  that  will  be  interchanged.

Denote by  $\epsilon_p $  the   sign  at  the   crossing  $p$, by $ \sigma_p TL$ the  tied  link  obtained  from  $TL$  by changing  the  sign  at the  crossing   $p$,  by $\tilde  \sigma_p TL$ the  tied  link  obtained  by $TL$  changing  the  sign  at    $p$ and  adding  a  tie  near $p$,  and  by  $\rho_p TL $ the  tied  link  obtained  by $TL$  removing  the  crossing     $p$ and  adding  a  tie.   Then,  do  the  same   for   point  $q$.

If     $q$  follows  $p$, then,  by  skein  relations Va,b,
$$
 \mathcal{F}(TL)=  w^{2\epsilon_p}[ \mathcal{F} ( \sigma_p TL) +  (u^{\epsilon_p}-1) \mathcal{F}(\tilde  \sigma_p TL)] + w^{\epsilon_p}(u^{\epsilon_p}-1 ) \mathcal{F}( \rho_p TL)
 $$
and
\begin{eqnarray*}
\mathcal{F}(TL)
& =  &  w^{2\epsilon_p} \{ w^{2\epsilon_q} [\mathcal{F}( \sigma_q \sigma_p TL )    +  (u^{\epsilon_q}-1) \mathcal{F}(\tilde \sigma_q \sigma_p TL)]+w^{\epsilon_q}(u^{\epsilon_q}-1) \F(\rho_q\sigma_p TL)  \}\\
&  &
  + (u^{\epsilon_p}-1)w^{2\epsilon_p}  \{ w^{2\epsilon_q} [\mathcal{F}( \sigma_q \tilde \sigma_p TL )    +  (u^{ \epsilon_q}-1) \mathcal{F}(\tilde \sigma_q \tilde \sigma_p TL)]+w^{\epsilon_q} ( u^{ \epsilon_q}-1 ) \F(\rho_q\tilde\sigma_p TL)  \}\\
&  &
 + (u^{ \epsilon_p}-1) w^{\epsilon_p} \{ w^{2\epsilon_q}[ \mathcal{F} ( \sigma_q \rho_p TL) +  (u^{ \epsilon_q}-1) \mathcal{F}(\widetilde  \sigma_q \rho_p TL)] + w^{\epsilon_q}( u^{ \epsilon_q}-1) \F(\rho_q  \rho_p TL)\}.
\end{eqnarray*}

If  $p$  follows  $q$,  then $\F(TL)$ is  obtained  from  the  above  expression  by interchanging  $p$ with  $q$. Observe that this expression contains  terms  of type $\F(\tau_q \tau_p TL)$  or $\alpha(\F(\tau_p \tau'_q TL) +\F(\tau'_p \tau_q TL))$,  where $(\tau,\tau')\in \{ \sigma, \tilde \sigma, \rho \}$ and  $\alpha$ is  a coefficient.  Such  terms  are   invariant  under  the  interchange of $p$ with $q$,  because   the  operation  $\tau_p$  commutes with  $\tau_q$  as  well  as with  $\tau'_q$. Therefore  $\F(TL)$ is  independent  of  the  order  of  $(p,q)$.

\end{proof}

\begin{remark}\label{circles} \rm
The necessity  of  II  in  the definition  of  $\mathcal{F}$ (Theorem \ref{the1})  is  due  to  the  fact  that  by  the sole skein  relation   we  cannot  calculate  the  value  of $\mathcal{F}$  on  $c$  unknotted  and  unlinked  circles,  without  ties  between  them.

The  value $X_c$ of  $\mathcal{F}$ on  $c$  unknotted  and  unlinked  circles,  all tied  together,  can be  calculated  using  skein  rule  IV recursively  $(c-1)$ times,  i.e.
\[ \frac{1}{u w } X_c- w X_c  = (1-u^{-1})  X_{c+1}, \quad \quad \text{from  which} \quad  \quad    X_{c+1}=\frac{t}{wz} X_c.\]
The  initial   value  $X_1=1$  of $\mathcal{F}$  is  given  by   rule I,  see  Figure \ref{ties4}.
\end{remark}

\begin{figure}[H]
\centering
\includegraphics{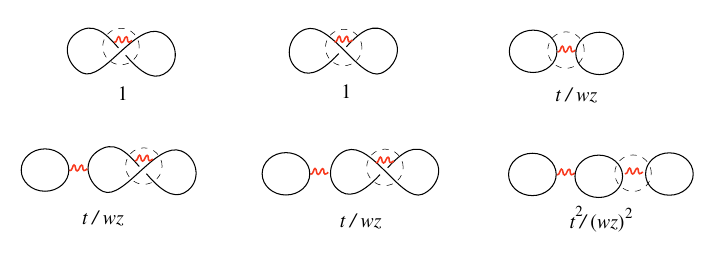}
\caption{The values  of  $\mathcal{F}$ on  the  unknotted and  unlinked tied circles   involved  in skein  rule IV.}\label{ties4}
\end{figure}
\begin{remark}\rm
Because of Remark \ref{LCT}, we observe  that  the  polynomial $\mathcal{F}$ provides  in  particular  an  invariant  polynomial  for  classical  links.
\end{remark}

\subsection{Properties  of  the  Polynomial  $\mathcal{F}$} \label{Properties}

Here  we  list  some  properties  of  the  polynomial  $\mathcal{F}$,  which  can  be  easily  verified.

\begin{enumerate}
\item[(i)]  $\mathcal{F}$  is  multiplicative  with respect to the  connected  sum of  tied  links
\item[(ii)] The  value  of  $\mathcal{F}$  does  not  change  if  the  orientations  of  all  curves  of  the  link  are  reversed
\item[(iii)] Let  $TL$ be a  link  diagram whose components  are all tied together, and $TL^\pm$ be the  link  diagram  obtained  from  $TL$  by changing  the  signs of  all  crossings. Thus  $\mathcal{F}(TL^\pm)$  is  obtained  from $\mathcal{F}(TL)$ by  the  following  changes:   $ w \rightarrow 1/w$  and    $u \rightarrow  1/u $
\item[(iv)] Let  $\Gamma$  be  a  knot or   a  link whose components  are all tied together, then    $\mathcal{F}(\Gamma)$  is  defined  by I and IV,  and     therefore satisfies the following  Homflypt--type  skein relation, cf. \cite{Lic}:
\begin{equation}\label{typehomflypt}
 \ell \mathcal{F}(TL_{+,\sim})  +  \ell^{-1} \mathcal{F}(TL_{-,\sim}) +  m  \mathcal{F}(TL_{\sim})= 0
\end{equation}
 where
 $$
 \ell = \frac{i} {\sqrt{u}w}   \quad \text{and} \quad m=i\left( \frac {1}{\sqrt u}-  \sqrt{u} \right).
 $$


\end{enumerate}

Item (i) is  deduced  from  the  defining  relation  I of  $\mathcal{F}$,  and  by  the  same  arguments that prove the  multiplicativity  of  the  invariants  obtained  by  skein  relations  (see  \cite{Lic}).  Item  (ii)  is  evident,  since the  value of  $\mathcal{F}$  on  the  unlinked  circles is  independent  of their  orientations,  and    the  skein  relations  are  invariant under  the  inversion  of  the  strands orientations.   Item  (iii) follows from  the  fact that, if  $ w \rightarrow 1/w$  and    $u \rightarrow  1/u $, the  skein  relations   Va and  Vb are  interchanged,  whereas the term $t/wz$ remains  unchanged.

  As  for item  (iv),  observe  that  if  the  components  of  the  links  are  all  tied together, or there is  a  unique  component,  then adding  a   tie  anywhere does not  alter the  isotopy class  of  the link; therefore,  in the  neighborhood  of  every crossing,  $TL_+$  (or $TL_-$)  can  be  replaced  by  $TL_{+,\sim}$  (or $TL_{-,\sim}$), and it can  be  treated by  the  sole   skein  relation  IV. Thus, it  is    evident  that  there  is  a  bijection  between the set of  isotopy  classes of   classical  links  and  the  isotopy  classes  of   tied  links  having  all components  tied  together. In  terms  of  diagrams,  it  is enough to add (to remove) a  tie      near  every  crossing  of  the diagram of the  classical  link,  as well  as  a tie between disjoint parts of  the diagram (see  Figure \ref{TL5}).

\begin{figure}[H]
\centering \includegraphics{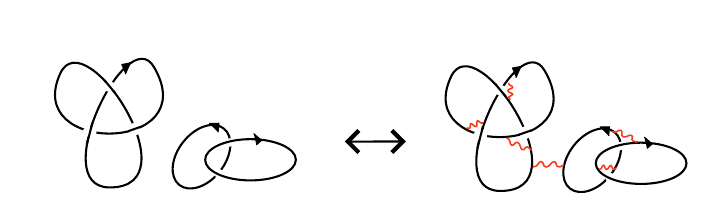}
\caption{Bijection between  classical  links  and  tied  links  with all  components tied  together.}  \label{TL5}
\end{figure}

Now,  multiplying  relation IV (Remark \ref{othersSkein}) by  $ u^{1/2}$ we obtain the skein relation (\ref{typehomflypt}), which is the defining  skein relation of  the  Homflypt polynomial.


\begin{remark}\rm
  By  writing   $w$ in  terms  of the  variables $u,z,t$ according  to  equation (\ref{tofw}): $$w=\sqrt{(z+t-ut)/uz}  $$   we observe  that  the  polynomial  $\mathcal{F}$  can  be  always  expressed  as  a Laurent polynomial  in  the  variables $u,z,t$,  multiplied for  $w^\epsilon$,  where   $\epsilon=1,0,-1$.   In  order  to  recover  the  polynomial  $\mathcal{F}$   by  means of the  Jones  recipe, it  turns  out  convenient  to  express  $\mathcal{F}$  in  this way.
\end{remark}

\subsection{Examples}

Let $H^+$,  $H^-$, $\widetilde{ H^+}$,  $\widetilde{ H^-}$,  $T^+$,  $T^-$ and  $E$  be the  tied  links  shown  in   Figure \ref{exam}.

\begin{figure}
\centering \includegraphics{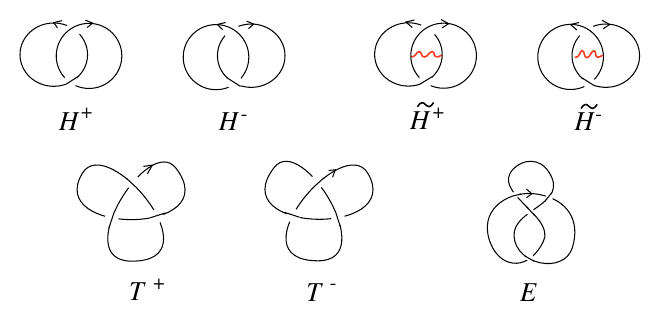}
\caption{}  \label{exam}
\end{figure}

We  show  the  calculation of  $\mathcal{F}$ for  the  first  link $H^+$;  for  the  others  we  write   the  result  of  the  calculation only.

Using  skein  rule Va, applied  to  the  upper crossing  of $H^+$,  we  write
$$  \mathcal{F}(H^+)= w^2 (\mathcal{F}(H_1)+(u-1)\mathcal{F}(H_2))+ w \mathcal{F}(H_3) $$
where  $H_1,H_2, H_3$  are  shown  in  the    figure below.

 \centerline{ \includegraphics{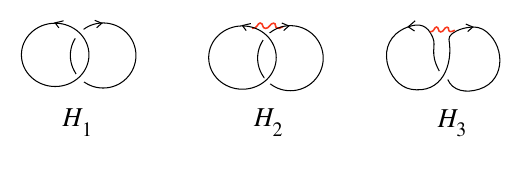}}

We  know  that,  by  rules  I  and  II,  and  by  Remark \ref{circles}:

$$  \mathcal{F}(H_1)= \frac{1}{wz},  \quad \quad \mathcal{F}(H_2)= \frac {t}{wz}, \quad  \quad  \mathcal{F}(H_3)=1.
$$
Therefore,
$$
\mathcal{F}(H^+)  =    \frac{w}{z}(1+ut +u z-t-z).
$$

For  the  other  links  of  Figure \ref{exam},  we  get:
$$
 \mathcal{F}(H^-)  =  \frac{u^2 +z+t-u z-ut}{ u w  (z+t-ut)},
$$\,
$$
\mathcal{F}(\widetilde {H^+})  = \frac{w}{z}( ut +u z- z),  \qquad \mathcal{F}(\widetilde{ H^-})  =   \frac{u^2 t+z+t-u z-ut}{ u w  (z+t-ut)},
$$\,
$$   \mathcal{F}(T^+)  =  \frac{ -u^3 tz- u^3t^2+2u^2  t^2  +3  u^2 t z +u^2 z^2-3  utz-u z^2-u t^2 +t z+z^2}{uz^2},
$$\,
$$
\mathcal{F}(T^-) =   \frac{ z(-u^3 t + u^2t -ut   +   u^2   z -u z  +z +t)}{u(z+t-ut)^2},
$$\,
$$
\mathcal{F}(E) =  \frac{    u^3 t^2 + u^3 t z -2 u^2t^2 -4 u^2 t z -u^2 z^2 +u t^2 +3 u z^2  +4 utz -z^2-t z }{uz(z+t-ut) }.
$$

 \begin{remark} \rm Observe  that  the  last  three  links  are  in  fact  knots (links  with one  component).  For  these  links, any  tie should     connect     the  component  with  itself,  i.e.  the  tie  is  not essential.
 \end{remark}

\section{The tied braid monoid }\label{Monoid}

The study  of classical links through the braid group is based on the classical theorems of  Alexander and Markov. The Alexander theorem states that any link can be obtained by closing a braid. The Markov theorem states when two braids yield  isotopic links.
 These theorems are repeated  for  singular knots \cite{bi, ge}, framed knots \cite{ks}, virtual knots \cite{kau} and $p$--adic framed links \cite{julaAM}. That is, in each of these classes of knotted objects,   a convenient  analogous of the braid group is  defined  and  analogous Alexander and Markov theorems  are  established. In this section we introduce the monoid of tied braids which plays the role of the braid group for tied links. Thus,  we will establish  the Alexander theorem and Markov theorem for tied links.

\subsection{}
We  introduce  now   the monoid of tied braids and we discuss   the diagrammatic interpretation for its defining generators.
\begin{definition}\label{monoidtb}\rm
 The tied braid monoid $T\!B_n$ is  the monoid generated by usual braids
 $\sigma_1, \ldots , \sigma_{n-1}$  and the generators  $\eta_1, \ldots ,\eta_{n-1}$, called ties,  such the $\sigma_i$'s satisfy braid relations among them  together with the following relations:
\begin{eqnarray}
\eta_i\eta_j & =  &  \eta_j \eta_i \qquad \text{ for all $i,j$}
\label{eta1}\\
\eta_i\sigma_i  & = &   \sigma_i \eta_i \qquad \text{ for all $i$}
\label{eta2}\\
\eta_i\sigma_j  & = &   \sigma_j \eta_i \qquad \text{for $\vert i  -  j\vert >1$ }
\label{eta3}\\
\eta_i \sigma_j \sigma_i & = &   \sigma_j \sigma_i \eta_j \qquad
\text{ for $\vert i - j\vert =1$}
\label{eta4}\\
\eta_i\sigma_j\sigma_i^{-1} & = &  \sigma_j\sigma_i^{-1}\eta_j \qquad \text{ for $\vert i  -  j\vert =1$}
\label{eta5}\\
\eta_i\eta_j\sigma_i & = & \eta_j\sigma_i\eta_j  \quad = \quad\sigma_i\eta_i\eta_j \qquad \text{ for $\vert i  -  j\vert =1$}
\label{eta6}\\
 \eta_i\eta_i & = & \eta_i \qquad \text{ for all $i$}.
 \label{eta7}
\end{eqnarray}
\end{definition}
In terms of diagrams the defining generator $\eta_i$
 corresponds to a  tie connecting   the $i$ with $(i+1)$--strands and $\sigma_i$ is represented as the usual braid diagram:

 \begin{figure}[H]
\centering \includegraphics{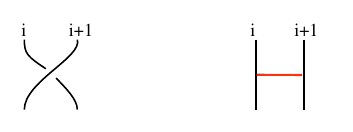}
\caption{Diagrammatic  representation  of  $\sigma_i$  and  $\eta_i$.}
\end{figure}

Now we  examine    the defining relations of $T\!B_n$ in  terms of  these  diagrammatic
interpretations of  the   defining generators.
Relations  (\ref{eta1})  and  (\ref{eta3}) are  trivial. Relation (\ref{eta2})  corresponds to the sliding of the tie through  the crossing.
Consider now  relations  (\ref{eta4})  and  (\ref{eta5}) in  terms  of  diagrams. The  first  one    corresponds to   the sliding of the tie from top to bottom
 behind or in  front of a strand. The  second  one  corresponds  to the  same  sliding  but bypassing the  strand.

\begin{figure}[H]
\centering \includegraphics{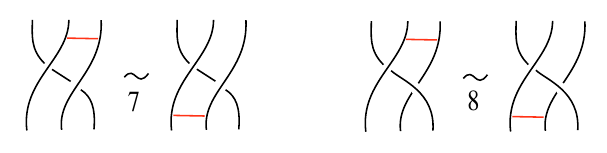}
\caption{Examples of relations (\ref{eta4})  and  (\ref{eta5}).}
\end{figure}

Finally,  we    see     relations  (\ref{eta6}) and (\ref{eta7})  in  diagrams.  Without  loss of  generality,  let us  assume  $i=2$, $j=1$ in (\ref{eta6}),  so  that

 \begin{equation}\label{c1} \eta_2\eta_1\sigma_2  =  \eta_1\sigma_2\eta_1   =  \sigma_2\eta_2\eta_1.      \end{equation}

\begin{figure}[H]
\centering \includegraphics{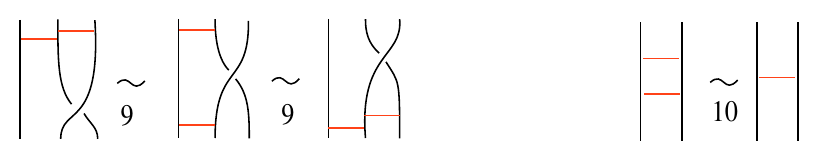}
\caption{Relations (9) and  (10). }  \label{TL3-6}\end{figure}

\subsection{Generalized ties}
Now we are going  to study certain  elements $\eta_{i,j}$ for all $\vert i-j \vert > 1$. These elements can be defined algebraically as follows,
$$
\eta_{i,j} =\sigma_i\cdots \sigma_{j-2}\eta_{j-1}\sigma_{j-2}^{-1}\cdots \sigma_{i}^{-1}
$$
see after proof of \cite[Lemma 1]{rh}.  In fact,  we want to analyze them in terms of diagrams; we will see that, diagrammatically, $\eta_{i,j}$ corresponds to a tie joining the  strand  $i$ with the  strand  $j$.  Furthermore, this generalized long tie    has the property that it is {\it transparent}  with respect to  all strands between the strands $i$ and $j$; i.e  it can be  drawn no  matter  if  in  front  or  behind  these  strands.  We  start by giving   a closer look to the following particular elements $\eta_{i,j}$:
\begin{equation}\label{eq1}
\eta_{i, i+2}:= \sigma_{i+1}^{-1}\eta_i\sigma_{i+1}.
\end{equation}
We   firstly multiply   both terms of relation (\ref{eta4}) at left by $ \sigma_j^{-1}$ and at right by $ \sigma_i^{-1}$, obtaining
\begin{equation}\label{eta12}
\sigma_j^{-1}\eta_i\sigma_j = \sigma_i\eta_j\sigma_i^{-1} \qquad \vert i-j\vert =1
\end{equation}
and  secondly we multiply  both terms of relation (\ref{eta4}) at left by $ \sigma_j^{-1}$ and at right by $ \sigma_j^{-1}$,   obtaining
\begin{equation}\label{eta21}
\sigma_j^{-1}\eta_i\sigma_j = \sigma_i^{-1}\eta_j\sigma_i \qquad \vert i-j\vert =1.
\end{equation}
Combining the last two equations we get
\begin{equation}\label{eq4}
\eta_{i, j+1}:=\sigma_j^{-1}\eta_i\sigma_j = \sigma_i\eta_j\sigma_i^{-1} =\sigma_j^{-1}\eta_i\sigma_j = \sigma_i^{-1}\eta_j\sigma_i \qquad \vert i-j\vert =1.
\end{equation}
In particular, for $n=3$, equations (\ref{eq4}) are expressed,  in terms of diagrams,  by:

\begin{figure}[H]
\centering \includegraphics{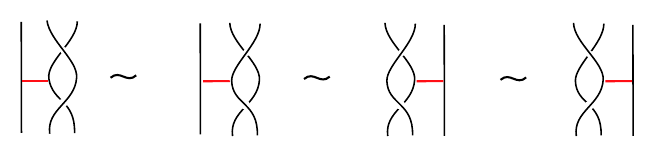}
\caption{ $\eta_{1,3}$  and  its  equivalent diagrams.} \label{TL3-3}
\end{figure}

Observe now that, if the ties are provided with elasticity, each one of the elements representing $\eta_{1,3}$ in Figure 10 can be transformed, by a Reidemeister move of second type in which the tie is stretched, in the following compact diagram (i.e., a tie connecting strand 1 with strand 3). From now on, the tie, having elastic property, will be represented as {\it a spring}.

\begin{figure}[H]
\centering \includegraphics{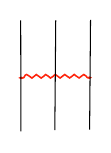}
\caption{A compact  diagram for $\eta_{1,3}$.}  \label{TL3-4}
\end{figure}

Note that  it  does  not  matter whether the  tie in the  compact  diagram of Figure \ref{TL3-4},   is    in  front  or  behind  the  strand 2.

Another  advantage  of  considering the  compact  diagram  (Figure \ref{TL3-4}) is  that  the  equation:
\begin{equation}\label{comm}       \eta_i  \sigma_j =  \sigma_j \sigma_j^{-1} \eta_i \sigma_j   =  \sigma_j  \eta_{i,j+1}\end{equation}
 can  be  interpreted  as  the  sliding  of  the  tie  up  and  down along  the  braid  under  stretching or  contracting. In  other  words,  while  the  elements  $\eta_i$  do  not  commute  with    $\sigma_j$ for  $|j-i|=1$,  the  equation  (\ref{comm})    can be interpreted  as a sort  of  commutation  between   $\sigma_i$  and  ties. The situation for $n=3$ is:

\begin{figure}[H]
\centering \includegraphics{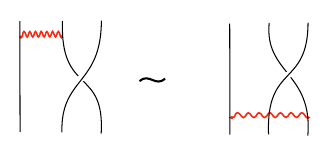}
\caption{$\eta_1 \sigma_2= \sigma_2 \eta_{1,3} $.}  \label{TL3-5}\end{figure}

  Let us continue regarding the case $n=3$ . By using  equations (\ref{comm})  and (\ref{eta1}),  we  obtain  from   (\ref{c1}) the  following  equalities
 \begin{equation}\label{c2} \sigma_2 \eta_{1,3}\eta_1= \eta_1\sigma_2\eta_1 = \eta_1 \eta_{1,3}  \sigma_2. \end{equation}

\begin{figure}[H]
\centering \includegraphics{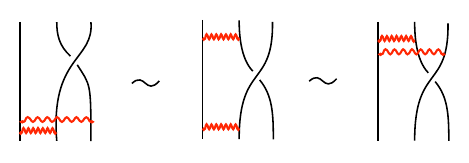}
\caption{   Relations(17).  }  \label{TL3-7}\end{figure}

Comparing now  (\ref{c1})  with  (\ref{c2})  we  get

\begin{equation}\label{EQ1}   \eta_{1,3} \eta_1= \eta_1 \eta_{1,3}= \eta_1 \eta_2.  \end{equation}
On  the  other  hand,  by using  equation  (\ref{eta6}) we  deduce
\begin{equation}\label{EQ2}   \eta_{1,3}  \eta_2=  \eta_2\eta_{1,3}=\eta_2 \eta_1 .  \end{equation}
  Having present (\ref{eta1}),   we  have  in terms of diagram the following equivalent  diagrams:

\begin{figure}[H]
\centering \includegraphics{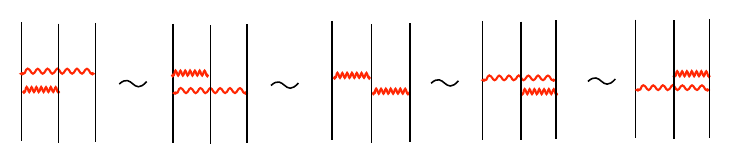}
\caption{  Left: relations    (\ref{EQ1});  right:  relations (\ref{EQ2}).  }  \label{TL3-8}\end{figure}

  Hence, in general for  $j-i=1$, we have that  $\eta_{i,j+1}$  commute  with  $\eta_i$, $\eta_j$  and
\begin{equation}\label{commeta}   \eta_{i,j+1} \eta_i = \eta_i\eta_{i,j+1}\eta_j =\eta_i\eta_j.  \end{equation}

We  are  now  ready  to  generalize  the  elements  $\eta_{i,j+1}$  ($j-i=1$),    to  the  elements  $\eta_{i,k}$,  for  every  $k\ge i$,  i.e.,
$\eta_{i,k}$  will  be  represented  by a  spring connecting  the  strand $i$  with  the  strand  $k$.
We  shall  say  also  that  such   tie  has {\sl length}  equal  to  $k-i$.

\begin{figure}[H]
\centering \includegraphics{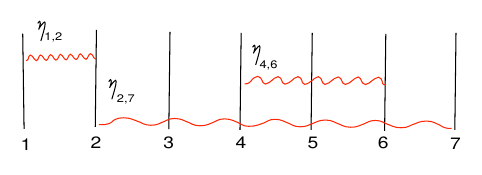}
\caption{Examples  of $\eta_{i,k}$.}  \label{TL3-9}\end{figure}

Of  course,  each  $\eta_i$ is a  tie of  length  1.
$$
\eta_{i,i+1}=  \eta_i.
$$
We  define  also the  tie  of  zero  length,  as the  monoid unit:

\begin{equation}\label{u}
\eta_{i,i}=1.
\end{equation}

In  virtue  of  the tie transparency,  there  are  different expressions  of  $\eta_{i,k}$  in  terms  of  the  $\sigma_j$  and  $\eta_j$.
An  example is  given  in   Figure  \ref{TL3-10},  where  the  three  diagrams are  all  equivalent  to  $\eta_{2,7}$.  The  proof  of  the  equivalence  is  based  on   equations (\ref{eta12})  and  (\ref{eta21}).

There  are  in  fact  $(k-i-1)2^{k-i-1}$ equivalent  expressions  of  $\eta_{i,k}$. Let   $s_i$  denote  either the  element  $\sigma_i$ or $\sigma_i^{-1}$,  and  let $\overline s_i=s_i^{-1}$.

Given  a pair $i,k$ such that  $k-i>1$,  the following   $2^{k-i-1}$  expressions of  $\eta_{i,k}$,  obtained for  all  possible choices  of  $s_l=\sigma_l$ or $s_l=\sigma_l^{-1}$:
 \[ \eta_{i,k} =s_i s_{i+1} s_{i+2} \cdots s_{k-2} \eta_{k-1}  \overline s_{k-2}  \cdots  \overline s_{i+1} \overline s_{i}   \]
are all  equivalent.  Moreover, for  every  $j$  such that  $i \le j < k-1$  there  are   similarly     $2^{k-i-1}$   equivalent  expressions:
$$
\eta_{i,k}= s_i s_{i+1} \cdots s_{j-1}  s_{k-1} s_{k-2} \cdots s_{j+1} \eta_j
 \overline s_{j+1} \cdots \overline s_{k-1}\overline s_{j-1} \dots \overline s_{i+1}\overline s_i.
$$

\begin{figure}[h]
\centerline{\includegraphics{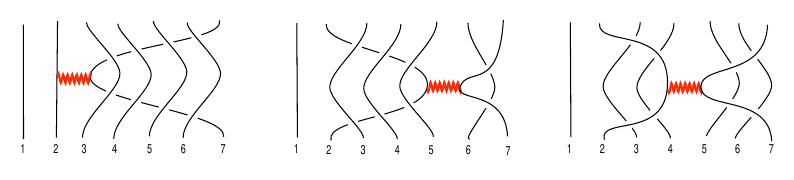}}
\caption{\small $\sigma_6^{-1} \sigma_5^{-1} \sigma_4^{-1} \sigma_3^{-1} \eta_2 \sigma_3 \sigma_4\sigma_5 \sigma_6 \sim
   \sigma_2\sigma_3 \sigma_4 \sigma_6  \eta_5 \sigma_6^{-1} \sigma_4^{-1}\sigma_3^{-1}\sigma_2^{-1} \sim \sigma_2^{-1} \sigma_3^{-1} \sigma_6 \sigma_5 \eta_4 \sigma_5^{-1} \sigma_6^{-1} \sigma_3 \sigma_2$.   }\label{TL3-10}
\end{figure}

Similarly, the  generalization  of (\ref{commeta})  to  all  $\eta_{i,k}$   reads (for  $i\le k\le m$) (see  examples  in  Figure  \ref{TL3-11}):
 \begin{equation}\label{simp}
 \eta_{i,k}\eta_{k,m}= \eta_{i,k}\eta_{i,m} = \eta_{k,m}\eta_{i,m}.
  \end{equation}
 In  particular,  if  $k=m$,  we  get
\[ \eta_{i,k}\eta_{k,k}= \eta_{i,k}\eta_{i,k}   \]
 i.e., by  (\ref{u}), all  ties  $\eta_{i,k}$  are  idempotent.  This  follows  as  well  from  any  expression  of  $\eta_{i,k}$,  in  virtue  of  (\ref{eta7}).

\begin{figure}[H]
\centering \includegraphics{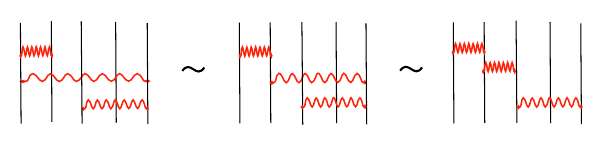}
\caption{Examples of relation (\ref{simp}):  $\eta_{1,2} \eta_{1,5} \eta_{3,5}= \eta_{1,2} \eta_{2,5} \eta_{3,5}=
\eta_{1,2} \eta_{2,3} \eta_{3,5} $.}    \label{TL3-11}\end{figure}

  Note  also  that  the  definition  of  $\eta_{i,k}$ for  every $k\ge i$  is  compatible  with the generalization to the  case  $k\le i$,  by  assuming,  for every  $k\ge i$
\begin{equation} \label{simmetry} \eta_{k,i}= \eta_{i,k}.
\end{equation}

  Before  concluding  this  section let  us see  the  generalization of equation  (\ref{comm}) to  generalized  ties of  any  length.

Let $s=\pm1$. Observe  the identities
\begin{equation}\label{commut1}
  \sigma_i^s   \eta_{i+1}=  \sigma_i^s  \eta_{i+1} ( \sigma_i^{-s} \sigma_i^{s})=  (\sigma_i^s  \eta_{i+1}   \sigma_i^{-s})  \sigma_i^{s} = \eta_{i,i+2} \sigma_i^s
\end{equation}
\begin{equation}\label{commut2}
  \eta_{i+1} \sigma_{i}^s = ( \sigma_{i}^{s} \sigma_{i}^{-s})   \eta_{i+1} \sigma_{i}^s=  \sigma_{i}^{s} (\sigma_{i}^{-s}    \eta_{i+1} \sigma_{i}^s)= \sigma_{i}^s \eta_{i,i+2}.
\end{equation}
They can  be  interpreted  as  the  sliding  of  the  tie  up  and  down along  the  braid  under  stretching or  contracting.
In  other  words,  while  the  element  $\sigma_i^{\pm 1}$     does  not  commute  with    $\eta_j$   when  $|j-i|=1$,  equations  (\ref{commut1}) and    (\ref{commut2}) provide   a  sort  of  commutation  rule between   $\sigma_i^{\pm 1}$    and  the spring.

Similarly,  a spring $\eta_{i,j}$ of  any  length bigger  than one, \lq commutes\rq\   (changing  its  length  by $\pm1$)   with   $\sigma_{i}$   and $\sigma_{i-1}$,  as  well as with  $\sigma_{j-1}$  and  $\sigma_j$,  according  to  the    equalities:
$$  \eta_{i,j}\sigma_{i}= \sigma_{i} \eta_{i+1,j},  \quad  \eta_{i,j} \sigma_{i-1} = \sigma_{i-1}  \eta_{i-1,j}, \quad \eta_{i,j}\sigma_{j}=\sigma_{j}\eta_{i,j+1}, \quad \eta_{i,j}\sigma_{j-1}\sigma_{j-1}= \eta_{i,j-1}.$$

 The  same  equalities  hold  for the  inverse of  the  generators $\sigma_i$'s.
  Therefore, in terms of diagrams, every   $\eta_{i,j}$  can be moved    to the bottom or to the top of  the  tied braid. More precisely, we have the following proposition.
\begin{proposition}[mobility  property]\label{down}
In any tied  braid all ties  can  be  moved to  the  bottom  (or to the top).   I.e.,  any  tied  braid    can  be put  in  the  form  $\beta \alpha$  (or $\alpha'\beta$),  where  $\beta$  is  a usual  braid,  and  $\alpha$ ($\alpha'$)  is a  set  of  generalized  ties.
\end{proposition}

  The generalized ties and its properties allow us to formulate the next propositions which play an important role in the next section.

 \begin{proposition}\label{equi} Any  set of  generalized ties in  $T\!B_n$   defines    an  equivalence relation   on  the    set of $n$  strands.
 \end{proposition}
\begin{proof}  The  properties  of  reflectivity,  symmetry  and  transitivity  follow  directly  from  equations (\ref{u}), (\ref{simmetry}) and  (\ref{simp}).
\end{proof}

\begin{proposition}\label{eqtb}  Let  $B_1$  and  $B_2$  be two  tied  braids  in  $T\!B_n$.  Let us
 write  $B_1=\beta_1 \alpha_1$  and  $B_2=\beta_2\alpha_2$   according  to  Proposition \ref{down}.  Then     $B_1=B_2$   if and only if   $\beta_1=\beta_2$  in  $B_n$  and $\alpha_1$  and  $\alpha_2$   define  the  same  partition  of  the  set   of  the   strands.
\end{proposition}

\subsection{The  Markov  and  Alexander  theorems  for  tied  links}

By taking the obvious monomorphism of monoids $T\!B_n$ into $T\!B_{n+1}$, we can   consider the inductive limit $T\!B_{\infty}$ associated to the inclusions chain $T\!B_2\subset \cdots \subset T\!B_{n-1}\subset T\!B_n\subset \cdots $. As in the classical case, given a tied braid $\tau$ we denote by $\widehat{\tau}$ its closure, which is  a tied link. We have then a map from $T\!B_{\infty}$ to $\mathcal{T}$. We are going to prove now that  in fact this map is surjective (Theorem \ref{alexandertb}). Later, we define the Markov moves for tied braids and then we prove a Markov theorem for tied links (Theorem \ref{markovtb}).

\begin{theorem}[Alexander theorem for tied links]\label{alexandertb}
Every oriented tied link can be obtained by closing  a tied braid.
\end{theorem}
\begin{proof}
Given a tied link $L(P)$, one  fixes  a   center $O$ in   the  plane of the  diagram of $L(P)$  and  proceeds  according to the Alexander procedure for  classical  links.  The  ties   do  not prevent  the  procedure  because  of  their  transparency,  so  that  the  tied  link  that we  obtain  is  isotopy  equivalent  to  $L(P)$. However, such  a  tied  link has   ties  connecting  pairs of  points  in  any  direction.   Using  the  property  that  the  ends  of  the  ties  can  slide  freely  along  the  strands of  the  link and  that  the  ties  are  transparent,  we arrange  them  so that  the  ends of  each  tie  lie  on one  half line originating  at  $O$,  not coinciding  with the  halfline  where  we  open  $L(P)$ (see  Figure \ref{TL19}).  The obtained  braid will  have horizontal   ties  connecting  two  points  of  different  strands. This  is by  construction a tied  braid  whose  closure  is  isotopy  equivalent  to  $L(P)$.
\end{proof}

\begin{figure}[H]
\centering \includegraphics{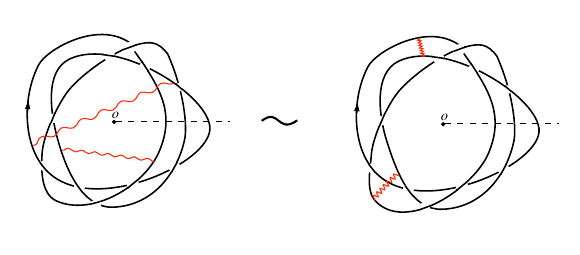}
\caption{Arranging  the  ties. }  \label{TL19}
\end{figure}

\smallbreak

 \noindent {\bf Notation.}
   Let  $\beta$  be a  tied  braid  in   $TB_{n}$. We  denote  $s_\beta$   the  permutation    associated, as usual, to the braid obtained  from $\beta$ by forgetting the ties in it.

\begin{definition}\label{simT}\rm

Two tied braids in  the monoid  $T\!B_{\infty}$ are $\sim_T$--equivalent if  one can  be  obtained  from  the  other    by a finite  sequence of     moves   belonging  to  the  following set of {\sl operations (or moves)}:
\begin{enumerate}
\item
$\alpha\beta$ \text{can be exchanged with} $\beta\alpha$
\item
$\alpha$ \text{can be exchanged with} $\alpha\sigma_n$ \text{or} $\alpha\sigma_n^{-1}$
\item $\alpha$ \text{can be exchanged with} $\alpha\eta_{i,j}$ if  $s_\alpha(i)=j$
\end{enumerate}
for all $\alpha, \beta\in T\!B_n$.

\end{definition}

\begin{theorem}[Markov Theorem for tied links]\label{markovtb}
 Two tied braids have    tie--isotopic   closures   if and only  if  they are $\sim_T$--equivalent.
\end{theorem}

 \begin{proof}   Let  $L_1$  and $L_2$  be  two  tie--isotopic tied  links,  and  $B_1$ and  $B_2$ the corresponding  tied  braids whose  closures  give  respectively  $L_1$  and  $L_2$, according  to  Theorem 2. We  have  to  prove  that  $B_1$  and  $B_2$  are  $\sim_T$--equivalent.    Firstly,  as in the  case  of  the  Alexander  theorem,
we can  proceed   for  tied  links  as in the  proof of  the  Markov  theorem for  classical  links,    using the     elasticity property,  transparency property,  and  the  fact  that
 the  ends of  the ties  can  freely  slide along  the  strands.   The  moves  (1)  and  (2)  of Definition \ref{simT}  coincide in  fact  with  the  classical  Markov  moves; further,     observe  that  the  ties  do  not  prevent these  moves.     By  consequence, the  braids   $\beta_1 $ and   $ \beta_2$, obtained  from  $B_1$  and  $B_2$ by forgetting  the  ties,  are Markov--equivalent.  Therefore,  we  can use  the  Markov  moves (1)  and  (2) to  transform   $\beta_1$  into  $\beta_2$.  In  this  way,     $B_1$  and  $B_2$,  after  these operations,  consist in the  same braid $\beta=\beta_2$  with $n$ strands,  to  which  ties  are  added  somewhere.  Since $L_1$ and  $L_2$  are  tie--isotopic,  in  particular they   have  the  same  set  $C$ of  components.
 Moreover,   the  ties of $L_1$ and  the ties  of  $L_2$ define  the  same partition of  $C$,  according  to    Definition \ref{isotop}.    Therefore, the  ties of  $B_1$  and  $B_2$,   under braids closure,  define  the same partition of $C$. However,  in  general,   the  ties in   $B_1$  do  not  coincide  with  those in  $B_2$. Even  worse, by  writing  $B_1=\beta \alpha_1$  and $B_2=\beta \alpha_2$ (see  Proposition \ref{down}),  in  general  $\alpha_1$  and  $\alpha_2$  do  not define  the  same partition  of  the  set of  $n$  strands, i.e.,  by Proposition \ref{eqtb}, $B_1\not=B_2$.  We  shall prove  that,  by  using  move (3),  we  can  put   $B_1$   and  $B_2$     in  an  equivalent  form.  Namely,    by using  repeatedly move (3),  we  add  at  the  bottom  of both  $B_1$  and $B_2$     the element  $\gamma$,  formed  by all the  elements  $\eta_{i,j}$,  for  every  $i=1,\dots,n$  such  that  $s_\beta(i)=j\not=i$.  I.e., we obtain  $B_k\sim_T  \beta \alpha_k \gamma$, for $k=1,2$   and  we  must  prove  that   $\beta\alpha_1 \gamma= \beta \alpha_2 \gamma$.  To  do  this,  it is  sufficient,  by Proposition \ref{eqtb},   to   prove  that  $\alpha_1 \gamma$  and  $\alpha_2 \gamma$  define  the  same  partition  of  the set  of  strands.
 Now,  observe that the cycles of  the  permutation  defined  by $\beta$  correspond to  the  components of  the original links.  Since $s_\beta(i)=j$ if and  only if  $i$  and  $j$ belong  to  the  same  cycle, hence   $\gamma$ contains $\eta_{i,j}$ if  $i$  and  $j$  belong to the  same  component of  the links.  I.e.,  every  subset in the strands  partition  defined by  $\gamma$  contains  strands  belonging to  a sole  links component.
 Remember  that   $\alpha_1$ and $\alpha_2$ contains  ties  that  define,  under  braid  closure,  the same partition  of  the  links components.  Therefore,  both $\alpha_1$  and  $\alpha_2$ do not  contain  necessarily  ties  connecting  the  same  component,  i.e., pairs of   strands contained  in  a  same  subset;  conversely, if  $\alpha_1$  contains a  tie  connecting  two  strands belonging to  two  different  subsets,  then  $\alpha_2$ must  contain a tie connecting two strands belonging to  the  same two subsets.  Therefore, every subset of the  partition of  the set of strands defined by both  $\alpha_1 \gamma$  and  $\alpha_2 \gamma$  is    either  a  subset, or the  union     the  same  subsets defined by  $\gamma$.  Hence  $\alpha_1\gamma=\alpha_2\gamma$.

\end{proof}

\section{Construction of $\mathcal{F}$ via the Jones recipe}

The procedure   to construct the Homflypt polynomial done in \cite{joAM}, leads to a generic way to construct an invariant of knotted objects, which is called  Jones recipe. The main objective  of this section is to  use the Jones recipe to construct the invariant $\mathcal{F}$, see Theorem \ref{Ftrace}. To do that,
 we   firstly note that  Theorems \ref{markovtb} and \ref{alexandertb}  allow to see  the set of  tied links as the set of  equivalence classes, under $\sim_T$, of $T\!B_{\infty}$. Secondly, in Proposition \ref{representation} below,  we define a representation of the  $T\!B_{\infty}$ in the  so--called  bt--algebra. This representation together with a  Markov trace,   supported by the bt--algebra, are the main ingredients in the Jones recipe for  the  construction of the invariant $\mathcal{F}$.

\smallbreak
{  \subsection{The bt--algebra  and  the  monoid  of  tied  braids }
In order to show  the first ingredient we shall recall the definition of the bt--algebra \cite{ju, AJ, rh, baART, AJtrace}.
Let $n$ be a positive integer. The bt--algebra, denoted $\mathcal{E}_n=\mathcal{E}_n(u)$,
is defined as follows.

\begin{definition}\label{btalgebra}\rm
  The algebra ${\mathcal E}_n$ is the associative unital ${\Bbb C}(u)$--algebra generated by $ T_1, \ldots, T_{n-1}$, $E_1, \ldots , E_{n-1}$ subject to  the following relations:
\begin{eqnarray}
\label{E1}
T_iT_j& = & T_jT_i  \qquad \text{ for all $\vert i-j\vert >1$} \\
\label{E2}
T_iT_jT_i& = & T_jT_i T_j \qquad \text{ for all $\vert i-j\vert = 1$} \\
\label{E3}
T_i^2 & = & 1 + (u-1) E_{i}  \left( 1+  T_i\right)\qquad \text{ for all $i$} \\
\label{E4}
E_iE_j & =  &  E_j E_i \qquad \text{ for all $i,j$}\\
\label{E5}
E_i^2 & = &  E_i \qquad \text{ for all $i$}  \\
\label{E6}
E_iT_i  & = &   T_i E_i \qquad \text{ for all $i$}\\
\label{E9}
E_iT_j  & = &   T_j E_i \qquad  \text{ for all    $\vert i - j\vert > 1$}\\
\label{E7}
E_iE_jT_i & = &  T_iE_iE_j \quad = \quad E_jT_iE_j\qquad \text{ for $\vert i  -  j\vert =1$}\\
\label{E8}
E_iT_jT_i & = &   T_jT_i E_j \qquad \text{ for $\vert i - j\vert =1$}.
\end{eqnarray}
\end{definition}

The relations (\ref{E3}) and (\ref{E5}) imply that $T_i$ is  invertible. Moreover,  we have
 \begin{equation}\label{inverse}
T_i^{-1} = T_i + (u^{-1}-1)E_i + (u^{-1}-1)E_i T_i.
\end{equation}

\begin{proposition}\label{representation}
The mapping $\sigma_i\mapsto T_i$ and $\eta_i\mapsto E_i$ defines a monoid representation, denoted $\varpi_n$, of $T\!B_n$ in  $\mathcal{E}_n$.
\end{proposition}

 \begin{proof}
 It is enough to prove that the defining relations of $T\!B_n$, suitably translated by replacing  $\sigma_i$ by $T_i$ and $\eta_i$ by $E_i$,  are satisfied in $\mathcal{E}_n$. Thus, it remains  only to check that the translation of relation (\ref{eta5}) holds in $\mathcal{E}_n$. From (\ref{inverse}), we have
$$
 E_i T_jT_i^{-1}  =  E_iT_jT_i + (u^{-1}-1)E_iT_jE_i + (u^{-1}-1)E_iT_jE_i T_i
$$
Now, having in mind relations (\ref{E8}) and  (\ref{E7}),  it follows that
\begin{equation}\label{trasparency}
 E_i T_jT_i^{-1} =   T_jT_i^{-1}E_i \qquad \text{for} \qquad \vert i-j\vert =1.
 \end{equation}
\end{proof}

\begin{remark}\rm
The proposition above says in particular that the  bt--algebra can be defined as the quotient of the monoid algebra of $T\!B_n$ by the two--sided ideal generated by the  elements $T_i^2 -1 -(u-1)E_i(1+T_i)$, for all $i$.
\end{remark}

\smallbreak
We shall recall now the second ingredient, that is, a Markov trace on the bt--algebra. Let  $\mathsf{A}$ and $\mathsf{B}$ be two   indeterminates. We have:

\begin{theorem}\cite[Theorem 3]{AJtrace}\label{trace}
There exists a family $\{\rho_n\}_{n\in\Bbb{N}}$ of Markov traces on the bt--algebra. I.e., for all $n\in {\Bbb N}$,   $\rho_n: {\mathcal E}_n\longrightarrow {\Bbb C}(u, \mathsf{A}, \mathsf{B})$ is the linear map  uniquely  defined by the following rules:
\begin{enumerate}
\item[(i)] $\rho_n(1) = 1$
\item[(ii)] $\rho_n(XY) = \rho_n(YX)$
\item[(iii)] $\rho_{n+1}(XT_n) =\rho_{n+1}(XE_nT_n)=\mathsf{A} \rho_n(X)$
\item[(iv)] $\rho_{n+1}(XE_n) =\mathsf{B} \rho_n(X)$
\end{enumerate}
where  $X, Y\in {\mathcal E}_n$.
\end{theorem}

In   \cite[Section 5]{AJtrace}, by using the Jones recipe,  we have defined an  invariant  polynomial $\bar{\Delta}$   for  classical  links.

$\bar{\Delta}$ is essentially the composition of the    natural  representation  of  the  braid  group   $B_n$  in $\E_n$ with the Markov trace above,  see \cite[Theorem 4]{AJtrace}. The invariant of tied links  that we will define now is nothing more than an extension of $\bar{\Delta}$ to tied links. To be precise,  we simply replace in the definition of $\bar{\Delta}$ the representation of the braid group in  ${\mathcal E}_n$ by the representation $\varpi_n$  of the tied braid monoid  in ${\mathcal E}_n$. Thus, we will denote also by $\bar{\Delta}$ this invariant of tied links.  More precisely, set
 \begin{equation}\label{ELLE}
 \mathsf{L} := \frac{\mathsf{A}+\mathsf{B}-u\mathsf{B}}{u\mathsf{A}}
 \end{equation}
and
\begin{equation}\label{D}
 \bar{D} = -\frac{1-\mathsf{L}u}{\sqrt{\mathsf{L}}(1-u)\B}
\end{equation}
so  that
\begin{equation}\label{LDA=1}
\sqrt{\mathsf{L}}\, \bar{D}\, \mathsf{A}= 1\quad  \text{or equivalently}\quad  \bar{D} =  \frac{1}{\A \sqrt{\L}}.
\end{equation}
 The invariant $\bar{\Delta}$ is thus defined as follows
\begin{equation}\label{norma}
\bar{\Delta}(\theta) := \bar{D}^{n-1} (\sqrt{\mathsf{L}})^{e(\theta )}(\rho_n\circ\varpi_n)(\theta)\qquad  (\theta\in T\!B_n )
\end{equation}
where  ${e(\theta )}$ denotes the exponent of  the  tied  braid $\theta$. That is,  if  $\theta = g_1^{e_1} g_2^{e_2} \dots g_m^{e_m}   \in  T\!B_n$, where  the $g_i$'s   are defining generators  of $T\!B_n$, then
\begin{equation}\label{eteta}
e(\theta) := \sum_{i=1}^N  \epsilon_i
\end{equation}
where  $\epsilon_i=e_i$ if  $g_i=T_k$  and  $\epsilon_i=0$ if  $g_i=E_k$.

\begin{theorem} \label{Ftrace}
 Let  $TL$  be a tied link   diagram obtained by  closing the tied braid $\theta \in T\!B_n$.  Let   $\A=z$,  $\B=t$.  Then
 $$  \bar\Delta(\theta)= \mathcal{F}(TL).  $$
\end{theorem}

\begin{proof}  We  use  here  the  results  of Section \ref{Sec2},  under  the  equations  $\A=z$  and  $\B=t$.   So, in  particular, $\L=w^2$.

If $TL$ is a  collection  of  $c$  unknotted,  unlinked  curves,  the  value  of  $\mathcal{F}(TL)$  is  $1/(\sqrt{\L}\A)^{c-1}$. Such  link  is indeed the  closure  of  the trivial  braid $\theta$ with  $c$  threads,  where  $\rho_c(\theta)=1$  and  $e(\theta)=0$.    Therefore $\bar\Delta(\theta)=\bar{D}^{c-1}=\mathcal{F}(TL)$. Observe  that  the  values of $\bar \Delta$ is  1  on  a  single  closed   unknotted curve.

 If $TL$ is a  collection  of  $c$  unknotted,  unlinked  curves with  $m$  essential  ties,   the  value  of  $\mathcal{F}(TL)$  is  $\B^m/(\sqrt{\L}\A)^{c-1}$. Such  link  is indeed the  closure  of  the   braid $\theta$ with  $c$  vertical threads,  of which  $m$  pairs  are  connected by  a  tie. Of  course  it  is  possible  to  arrange  the  ties  so  that  they have all  length  one  and  connect  the  last  $m+1\le c$  threads.  Then,  using item (iv) of the  definition  of  the  trace, we  obtain   $\rho_c(\theta)=\B^m$.  Moreover,     $e(\theta)=0$.    Therefore $\bar\Delta(\theta)=\bar{D}^{c-1} \B^m =\mathcal{F}(TL)$.

 Suppose  now  that  four  tied  braids in  $\E_n(u)$  are  given,  $\theta_+,  \theta_-, \theta_\sim$  and  $\theta_{+,\sim}$,  that  are  all  identical  except  for  the  neighborhood  of  a  $T_i$  element,  exactly  as  for  the  tied  links, see  Figure  \ref{ties3}.    Now,    formula (\ref{inverse}) of the inverse of $T_i$   and the linearity of the trace imply that:
\begin{equation} \label{preskein}
\rho_n (\theta_+) - \rho_n(\theta_-) =  (1-u^{-1}) \rho_n (\theta_\sim) + (1-u^{-1}) \rho_n (\theta_{+,\sim}).
\end{equation}
Let  now  $TL_+,  TL_-, TL_\sim$  and  $TL_{+,\sim}$ be the  tied links  obtained  from  the  closure  of  the  tied  braids  above.  The  polynomial  $\bar \Delta$
for  these  links  is  obtained  multiplying  the  trace  by  the  factor  $\bar D^{n-1}  (\sqrt{\mathsf{L}})^{e(\theta_k )}$  ($k=+,\  - ,\ \sim$  or  $+,\sim$).
Now,  let  us  denote   $e(\theta_\sim)=d $.  By  the  definition it is  evident  that  $e(\theta_+)=e(\theta_{+,\sim})= d +1$  and     $e(\theta_-)= d-1$.  Therefore  (\ref{preskein})  can  be  written  in  terms  of  the  polynomial  $\bar\Delta$:
\begin{equation}
\frac{1}{\sqrt{\L}} \bar \Delta(\theta_+) -  \sqrt{\L}  \bar \Delta (\theta_-) =  (1-u^{-1})    \bar \Delta (\theta_\sim) + (1-u^{-1}) \frac{1}{\sqrt{\L}} \bar \Delta (\theta_{+,\sim})
\end{equation}
which  coincides with  the  skein  rule  III  for  the  polynomial  $\mathcal{F}$.

Since  $\bar\Delta$ is  a  topological  invariant  for  links  (see  \cite{AJtrace})  and  satisfies  the  same  rules  I, II and III  as  the  polynomial  $\mathcal{F}$, invariant  for tied  links,  it  coincides  with  $\mathcal{F}$.

\end{proof}
\subsection{Examples}

We  calculate  here  the polynomial  $\bar \Delta$  for  some  tied  links  shown  in  Figure \ref{exam}.  We  will  see   that,   renaming  the  variables,   $\bar \Delta=\mathcal{F}$.

The Hopf link $H^+$  is  the  closure  of  the  braid  $T_1^2$. Therefore, using  (\ref{E3}),
$$\rho_2(T_1^2)=\rho_2(1+(u-1)E_1+(u-1)E_1T_1)=1+(u-1)\B +(u-1)\A $$
and,  using  (\ref{norma})  with  $n=2$  and  $e(T_1^2)=2$ we have
$$  \bar \Delta(T_1^2)= \frac {\sqrt{\L}}{A} (1+(u-1)\B +(u-1)\A). $$

Similarly, $\tilde H^-$  is  the  closure  of  the  braid  $E_1T_1^{-2}$. Therefore, using  (\ref{inverse})  and  (\ref{E3}), we  get
$$\rho_2(E_1T_1^{-2})=\frac{u^2\B-u\B+\B-u\A+\A}{u^2}.$$
Here  $n=2$  and  $e(E_1T_1^{-2})=-2$,  therefore
$$  \bar \Delta(E_1T_1^{-2})=\frac{u^2\B-u\B+\B-u\A+\A}{u(\A+\B-u\B)\sqrt{\L}}.$$

The  knot  $E$  is  the  closure  of the  braid  $T_1T_2^{-1}T_1T_2^{-1}$.  To  calculate  the  trace  we  use the  algorithm  shown  in  the  next  section.
$$\rho_3(T_1T_2^{-1}T_1T_2^{-1})=  \frac{ (u^3-4u^2+4u -1)\A\B+(3u-1-u^2)\A^2 + (u^3-2u^2+u)\B^2   }{u^2}$$
and,  using  (\ref{norma})  with  $n=3$  and  $e(T_1T_2^{-1}T_1T_2^{-1})=0$ we obtain
$$  \bar \Delta(T_1T_2^{-1}T_1T_2^{-1} )= \frac{ (u^3-4u^2+4u -1)\A\B+(3u-1-u^2)\A^2 + (u^3-2u^2+u)\B^2 }{u\A(\A+\B-u\B)}.   $$

\section{Computer  computation  of $\mathcal{F}$  }

In  this  section  we  show  how  to  calculate  the  polynomial $\mathcal{F}$  for  a  tied  link  or  a  classical link  by means of  the Theorem \ref{Ftrace}.  Indeed,  if  a   (tied) link $L$ is  put  in  the  form of  (tied) braid $X$,  then   we calculate   the  trace of  $X$,  and  then we  normalize  it by (\ref{norma}) to  obtain  the  invariant $\mathcal{F}$.

An  element  of the  algebra $\E_n$ is a linear  combination of  words, i.e.,  finite expressions  in  the generators $T_1, \dots, T_{n-1}$,  and  the  $E_1, \dots, E_{n-1}$.  The  coefficients
are Laurent polynomials in the parameter $u$.   An {\sl addend}  is  a single word with a  coefficient.

A word   is  {\sl  simple}  if  the  consecutive  generators   in  it  are  different  and  appear  to  the  first power.

The  trace  of  an  element  of  the  algebra     is obtained  as  a   linear  combination  of  traces  of  words.   A word  of  $\E_n$  containing  a  sole  element  in  the  set  $\{E_{n-1},T_{n-1},E_{n-1} T_{n-1}\}$  is  said  {\sl $\rho$--reducible}.  Indeed, it  is  reduced  by  the  trace  properties, stated  in  Theorem 3, to a  coefficient  times  the trace  of  a  word  of $\E_{n-1}$.  Therefore,  to  calculate  the  trace  of  a  word,  one  needs  to transform  every  word  into  a  word  or  a  linear  combination  of  words   $\rho$--reducible.

We list  here  a  series of procedures used  by  the  algorithm.

 \subsection{Simplification  of an addend} \label{proc1}   Iterations  of  the  following   procedures  reduce  a  word  into  a  linear  combination  of  simple  words.
\begin{enumerate}

 \item[\bf{S1}]    $k$ consecutive  copies  of  the  same  generator $E_i$
are replaced  by a unique  $E_i$  because of  the   relation
$$
 E_i^k =E_i  \quad \ \hbox{for every  $k>0$}.
$$

\item[\bf{S2}]       Consecutive powers  of the  same  generator $T_i$
 are  replaced by $T_i$   to the  algebraic sum of the  exponents of such  powers.

\item[\bf{S3}]
   An  addend  whose  word  contains  powers  of the $T_i$'s
  with exponents different  from one is  transformed  into a  sum of addends  containing  $T_i $'s to the first  power.
   We use the following relations that  follows  from  relations (\ref{E3}), (\ref{E5}),  (\ref{E6}) and  (\ref{inverse}).
  \begin{eqnarray*}  & T_i^{2m}  = 1+ \sum_{k=0}^{2m-1} (-1)^{k+1} u^k (E_i+E_iT_i)   \quad \ \hbox{for every  $m>0$}  \\
    & T_i^{2m+1}  = T_i+ \sum_{k=1}^{2m} (-1)^{k} u^k (E_i+E_iT_i)    \quad \ \hbox{for every  $m>0$} \\
    & T_i^{-2m}  = 1+ \sum_{k=0}^{2m } (-1)^{k } u^{-k} (E_i+E_iT_i)     \quad \ \hbox{for every  $m>0$}  \\
    & T_i^{-(2m+1)}  =  T_i+ \sum_{k=1}^{2m+1} (-1)^{k+1} u^{-k} (E_i+E_iT_i)   \quad \ \hbox{for every  $m>0$}. \end{eqnarray*}

\end{enumerate}

\subsection{Reduction  of  a  word  }\label{proc2}

The  following  procedures are  used to  make  a  word  $\rho$--reducible

\begin{enumerate}
\item[\bf{R1}]    Denote by  $X_i,Y_i,Z_i$     elements   in  the set
\begin{equation}\label{Gi} G_i:=\{T_i,  E_i,  E_iT_i =T_iE_i\}. \end{equation}
A word  of  type $X_i Y_{i-1} Z_i$  is  said
{\sl  reducible}.
   Using    the defining relations of the bt--algebra,  this   procedure   transforms  a  reducible word  $X_i Y_{i-1} Z_i$  into  a  word  or into a  linear  combination  of  words containing one  and  only one  element  of  $G_i$.   There  are
  in all  27    cases,  listed  here:

 \begin{eqnarray*}
E_i E_{i-1} E_i   & \rightarrow &  E_{i-1} E_i \\
 E_i T_{i-1} E_i   & \rightarrow & T_{i-1} E_{i-1} E_i \\
  E_i T_{i-1} E_{i-1}E_i   & \rightarrow & T_{i-1} E_{i-1} E_i \\
 E_i  E_{i-1}T_{i}   & \rightarrow & E_{i-1} T_{i} E_i \\
E_i  T_{i-1}T_{i}   & \rightarrow & T_{i-1} T_{i} E_i \\
E_i T_{i-1}E_{i-1}T_i   & \rightarrow &   T_{i-1}E_{i-1}T_i E_i \\
E_i E_{i-1} T_i E_i    & \rightarrow &   E_{i-1}T_i E_i \\
E_i T_{i-1}T_i E_i   & \rightarrow &   T_{i-1}E_{i-1}T_i E_i \\
E_iT_{i-1}E_{i-1}T_i E_i    & \rightarrow &   T_{i-1}E_{i-1}T_i E_i \\
T_i E_{i-1}E_i    & \rightarrow &  E_{i-1}T_i E_i \\
T_i T_{i-1}  E_i    & \rightarrow &  E_{i-1}T_i T_{i-1} \\
 T_i T_{i-1}E_{i-1}  E_i     & \rightarrow &   E_{i-1}T_i E_iT_{i-1} \\
  T_i E_{i-1}T_i  & \rightarrow & T_{i-1} E_i T_{i-1} \ + \ (1-u)\   E_i T_{i-1}E_{i-1} \ + \ (u-1) \ T_i  E_i E_{i-1} \\
  T_i T_{i-1}T_i   & \rightarrow &  T_{i-1}T_i T_{i-1} \\
  T_i T_{i-1}E_{i-1}T_i   & \rightarrow &   T_{i-1}T_i E_iT_{i-1} \\
  T_i E_{i-1}  T_i E_i   & \rightarrow &  u \ E_{i-1}  E_i \ +\  (u-1)\  E_{i-1}  T_i  E_i \\
  T_i T_{i-1}  T_i E_i   & \rightarrow &   T_{i-1}E_{i-1}T_i T_{i-1} \\
  T_i T_{i-1}E_{i-1}  T_i E_i    & \rightarrow &   T_{i-1}E_{i-1}T_i E_iT_{i-1} \\
  T_i  E_iE_{i-1}E_i   & \rightarrow &   E_{i-1}T_i E_i \\
  T_i E_iT_{i-1}E_i    & \rightarrow &  E_{i-1}T_i E_iT_{i-1}\\
  T_i E_iT_{i-1}E_{i-1}E_i    & \rightarrow &   E_{i-1}T_i E_iT_{i-1} \\
  T_i E_iE_{i-1}T_i   & \rightarrow &   u \  E_{i-1} E_i \  + \   (u-1)\  E_{i-1} T_i  E_i \\
  T_i E_iT_{i-1}T_i   & \rightarrow &   T_{i-1}T_i T_{i-1}E_{i-1} \\
  T_i E_i T_{i-1}E_{i-1}T_i    & \rightarrow &   T_{i-1}E_{i-1}T_i E_iT_{i-1} \\
   T_i E_iE_{i-1}T_i E_i   & \rightarrow &  u \  E_{i-1} E_i \  +  \ (u-1)\  E_{i-1} T_i  E_i \\
  T_i E_i T_{i-1}T_i E_i   & \rightarrow &  T_{i-1}E_{i-1}T_i E_iT_{i-1} \\
  T_i E_iT_{i-1}E_{i-1} T_i E_i   & \rightarrow &   T_{i-1}E_{i-1}T_i E_iT_{i-1}.
  \end{eqnarray*}


\end{enumerate}


Let  $m$  be  the  maximum  index  of  the  elements of the simple  word  $X$.      We  suppose, moreover, that  $X$  contains  $r>1$  elements  from   $G_m$ (defined by (\ref{Gi})),  so  that  it is  not  $\rho$--reducible.  The  iteration   of  the  following  procedures transforms  $X$  into  a word  or  a  linear  combination  of  words  with $r-1$  elements  from  $G_m$.

  Let  $X=x_1 x_2\dots x_t$,  where  each $x_j$  is  an  element  from  a $g_k$,  $k\le m$,  and  let  $g_m$ be      the  first  element  from  $G_m$  encountered  in  the  simple  word  $X$.   Observe  that  $X$  is  simple,  so  that   the  successive  element to  $g_m $   belongs  to  $G_k$, $k\le m-1$.  Write
  $$  X= X_1 g_m   Y_1. $$

   {\bf STEP 1}
 \begin{enumerate}

\item[{\bf  R1}]   Let  $y$  be  the  first  element  of  $Y_1$. If $y=g_i$,  with $i<m-1$,  and  let  $X_2=X_1  y$  and  $Y_1= y Y_2$, so  that
$$   X= X_2  g_m Y_2 . $$  Then  $X_1\leftarrow X_2$  and  $Y_1\leftarrow Y_2$  and  repeat, until  $y\in  G_{m-1}$  or  $y\in G_m$. If  $y\in  G_m$,  then  the  simplified  $X$   is  a    word  or  a  linear  combination  of  words  with at  most  $r-1$  elements  from  $G_m$. Otherwise

\item[{\bf  R2}] If  $y=g_{m-1}\in  G_{m-1}$,    write  $Y_1=g_{m-1}Y_2$,    so  that
$$  X= X_1  g_m g_{m-1} Y_2 . $$  Then   $Y_1\leftarrow Y_2$.  Go  to  step 2.
\end{enumerate}

{\bf  STEP 2}

\begin{enumerate}
\item[{\bf  R4}]  If the  first  element $y$  of  $Y_1$     belongs  to  $G_m$,  then  reduce  the  word by  R1,  so  that the  simplified  $X$   is  a    word  or  a  linear  combination  of  words  with at  most $r-1$  elements  from  $G_m$. Otherwise

\item[{\bf  R5}]  If $y=g_i$,  with $i<m-2$,    let  $X_2=X_1  y$  and  $Y_1=y Y_2$, so  that
$$   X= X_2  g_m g_{m-1} Y_2 . $$  Then  $X_1\leftarrow X_2$  and  $Y_1\leftarrow Y_2$.  Let  $y$   be the  first  element  of  $Y_1$. If $y=g_i$,  with $i<m-2$, then   repeat  R5. If $z\in  G_m$,  then   go  to  R4. If  $y\in G_{m-1}$,  then simplify.  The  simplified  words  are  of type  $X= X_1 g_m Y_1$ (in  that  case  go  to  R2)  or  of  type      $  X= X_1  g_m g_{m-1} Y_1 $ (in  that  case  go  to  R4).  Otherwise

\item[{\bf  R6}]  If $y=g_i$,  with $i=m-2$,    write  $Y_1=g_{m-2}Y_2$,    so  that
$$  X= X_1  g_m g_{m-1}g_{m-2} Y_2 . $$  Then   $Y_1\leftarrow Y_2$.   Go  to  next step.

\end{enumerate}
   {\bf  STEP  n}
   At  step  $n$  every  simple  word  $X$  with  $r>1$  elements  from  $G_m$  is  written  as
   $$  X=  X_1  Z_{m,n} Y_1 \quad  \text{where} \quad Z_{m,n}= g_m g_{m-1}g_{m-2} \dots g_{m-n+1}.$$
   \begin{enumerate}
\item[{\bf  R7}] Let  $k=m-n+1$.  If  the  first  element $y$ of  $Y_1$  is  $g_i$  with  $i<k-1$,  then   put it at  the  end  of $X_1$, since  it  commutes  with  every  $g_j$,   $k\le j\le m$, and     proceed by  analyzing  the  successive  element.
 \item[{\bf  R8}] If  $y=g_j$,  with  $j=k$,  then  simplify  the  word.  The  simplified  words  have  to  be  processed  by  step  $n$  or  $n-1$.
 \item[{\bf  R9}]  If $y\in G_j$, with  $k<j< m$,  then  $y$  commutes  with all  elements   of  $Z_{m,n}$  with index  less  than  $j-1$,  so  that,  writing $y=g'_j$,  we  get
     $$  X=  X_1  g_m g_{m-1} \dots  g_{i+1} g_{j} g_{j-1} g'_j  g_{j-2} \dots g_k. $$
 The  subword   $g_{j} g_{j-1} g'_j$    is  processed  by  R1,   replacing  it  by   subwords of  type    $ g'_{j-1} g''_j$  or $g'_{j-1} g''_j g''_{j-1}$.
Since  $g'_{j-1}$  commutes   with  $g_{i+1},g_{i+2},\dots, g_m$,  it  is  put  at  the  end  of $X_1$.    If  the  element  $g''_{j-1}$  is  absent,  then
the  subword  $g_{j-2} \dots g_k$  commutes  with  $g_m  g_{m-2} \dots g_j$  and  it is put  at  the  end  of $X_1$.  Go  to step  m-j+1. Otherwise,  go  to  R8.
\item[{\bf  R10}]   If $y\in G_m$,    then  $y$  commutes  with all  elements   of  $Z_{m,n}$  with index  less  than  $m-1$,  so  that,  writing $y=g'm$,  we  get
     $$  X=  X_1  g_m g_{m-1}   g'_{m} g_{m-2}   \dots  g_k. $$ The  word  is  reduced by  R1  and the  reduced words  have  at  most  $r-1$  elements  from  $G_m$.
\item[{\bf  R11}]    If $y=g_{k-1}$,   write  $Y_1=g_{k-1}Y_2$,    so  that
$$  X= X_1  g_m g_{m-1}g_{m-2}\dots \dots g_{m-n }  Y_2 . $$  Then  $Y_1\leftarrow Y_2$.   Go  to  step  n+1.
\end{enumerate}

Since  the  number  of  elements  between  $g_m$  and  the  second  occurrence in  $X$  of  an  element  from  $G_m$  is  finite,  the  case $y\in  G_m$   happens  for  some  step $n\ge 1$,  so  that  the  number  $r$ of  occurrences  of  elements  from  $G_m$  is  diminished.  When  $r=1$     the  word  is  $\rho$--reducible.

\subsection{Contraction }\label{CONTR}
Suppose  the  word  $X$  be  $\rho$--reducible.   Then  we  write  $X$  as   $W_{m-1 } g_m V_{m-1}$, with  $g_m\in  G_m$.

\begin{enumerate}

\item[{\bf C1}]   The  words   $W_{m-1 }$  and $V_{m-1}$  are  two  simple words of $\E_{m}$. Let   $\alpha$  be a coefficient.  This  procedure
transforms  the  following simple addends:
\[ \begin{array}{lll} \alpha \  W_{m-1}E_m V_{m-1} &     \rightarrow  &   \alpha \  \B \   W_{m-1}V_{m-1}  \\
\alpha \   W_{m-1}T_m V_{m-1} &     \rightarrow &   \alpha \ \A \   W_{m-1}V_{m-1}  \\

\alpha \  W_{m-1}E_m T_mV_{m-1}     & \rightarrow &  \alpha \ \A \  W_{m-1}V_{m-1}. \\
\end{array}\]

This procedure applies  the properties  of the trace.  Indeed, we  have  for instance  (see statements (i),(ii) and (iii) in Theorem  \ref{Ftrace}),
 \[ \rho_{m+1}( \alpha \  W_{m-1} T_m V_{m-1})= \text{(by (ii))} = \alpha \  \rho_{m+1}( V_{m-1} W_{m-1} T_m ) = \quad \quad  \text{(by (iii))} \]
 \[ = \alpha \ \A \ \rho_m ( V_{m-1} W_{m-1})= \text{(by (ii))} = \alpha \  \A \  \rho_m ( W_{m-1} V_{m-1}). \]
  When  $m=1$,  $W_{m-1}=V_{m-1}=1$, so  that $\rho_m ( W_{m-1} V_{m-1})=1$ (by   (i)).
    Therefore the  output of  the procedure
is a polynomial  and  increments the trace.

\end{enumerate}

\vskip  1 cm

  \noindent {\bf Remark}:
    Recently, the bt--algebra and  the  tied  links were used in the definition  of a new
invariant of classical links,  for  details  see  \cite[Section 8]{cjkl}.    Finally,  we  notice  that
I. Marin in \cite{marin} has associated to every Coxeter group  a certain   algebra    that, when
the Coxeter group is finite of type $A$,   coincides  with the  bt--algebra.

\section*{Acknowledgments}
  We  are  grateful to the Referees  for their   valuable
comments and suggestions.


\begin{thebibliography}{25}



\bibitem{AJ} F. Aicardi, J. Juyumaya, {\rm An algebra involving braids and ties.} Preprint ICTP   IC/2000/179, Trieste.

\bibitem{AJtrace} F. Aicardi, J. Juyumaya, {\rm Markov trace on the algebra of braids and ties.}   {\em  Moscow Math. J.} {\bf 16}   (2016), no. 3,  1--35.

\bibitem{Ai} F. Aicardi, {\rm An  invariant of colored  links via skein  relation}.  {\em  Arnold  Math. Journal}, March 2016 {http://link.springer.com/article/10.1007/s40598-015-0035-1/fulltext.html}

\bibitem{ba} J.C. Baez, {\rm Link invariants of finite type and perturbation theory},
  {\em Lett. Math. Phys.} {\bf 26} (1992),  no. 1, 43--51.

\bibitem{baART} E. O. Banjo,  {\rm The generic representation theory of the Juyumaya algebra of braids and ties}.
 {\em Algebr. Represent. Theory} {\bf 16} (2013), no. 5, 1385--1395.

\bibitem{bi} J.S. Birman, {\rm New points of view in knot theory}, {\em Bull. Amer. Math. Soc. (N.S.)} {\bf 28}  (1993),  no. 2, 253--287.

\bibitem{cjkl} M. Chlouveraki, J. Juyumaya, K. Karvounis, S. Lambropoulou, {\rm Identifying the invariants for classical knots and links from the Yokonuma--Hecke algebras}. See arXiv:1505.06666.

\bibitem{ge} B. Gemein, {\rm Singular braids and Markov's theorem}, {\em J. Knot Theory Ramifications} {\bf 6} (1997), no. 4, 441--454.

\bibitem{joAM} V.F.R. Jones, {\rm Hecke algebra representations of braid groups and link polynomials},
 {\em Ann. Math.} {\bf 126} (1987), 335--388.

\bibitem{ju} J. Juyumaya, {\rm Another algebra from the Yokonuma-Hecke
algebra}. Preprint ICTP, IC/1999/160.

\bibitem{julaAM} J. Juyumaya, S. Lambropoulou, {\rm p--adic  framed braids II}, {\em Adv.   Math.} {\bf 234} (2013), 149--191.




\bibitem{kau} L. Kauffman,{\rm Virtual Knot Theory}, {\em European J. Combin.} {\bf 20} (1999), 663--690.


\bibitem{ks} K.H. Ko, L. Smolinsky, {\rm The framed braid group and $3$--manifolds},
{\em Proceedings of the AMS}, {\bf 115}, no. 2, 541--551 (1992).

\bibitem{Lic} W.B.R. Lickorish,  K.C.  Millett,   {\rm A Polynomial  Invariant  of  Oriented Links.}  {\em Topology}, {\bf  26},  No. 1 (1987), 107--141.

\bibitem{marin} I. Marin, {\rm Artin groups and Yokonuma--Hecke
algebras}. See arXiv: 1601.03191.



\bibitem{rh} S.  Ryom--Hansen, {\rm On the representation theory of an algebra of braids and ties.} {\em J. Algebr. Comb.} {\bf 33}  (2011),  57--79.


\bibitem{sm} L. Smolin, {\rm Knot theory, loop space and the diffeomorphism group}, {\em New perspectives in canonical gravity},  245--266, Monogr. Textbooks Phys. Sci. Lecture Notes, {\bf 5}, Bibliopolis, Naples, 1988.


\end{thebibliography}
\end{document}